\documentclass[12pt]{amsart}
\usepackage{etex}

\usepackage{amsthm,amssymb,enumitem,mathtools}

\usepackage{tikz}

\usepackage{scalerel}

\usetikzlibrary{arrows,shapes,automata,backgrounds,decorations,petri,positioning}
\usetikzlibrary[calc,intersections,through,backgrounds,arrows,decorations.pathmorphing]

\tikzstyle{edge} = [fill,opacity=.5,fill opacity=.5,line cap=round, line join=round, line width=50pt]

\pgfdeclarelayer{background}
\pgfsetlayers{background,main}

\usepackage[margin=1in,letterpaper,portrait]{geometry}

\usepackage{amsthm}

\usepackage[latin1]{inputenc}
\usepackage{subfigure}
\usepackage{color}
\usepackage{amsmath}
\usepackage{amsthm}
\usepackage{amstext}
\usepackage{amssymb}
\usepackage{amsfonts}
\usepackage{graphicx}
\usepackage{young}
\usepackage{multicol}
\usepackage{mathrsfs}
\usepackage[all]{xy}
\usepackage{tikz}
\usepackage{mathtools}
\usepackage{tabularx}
\usepackage{array}
\usepackage{commath}
\usepackage{ytableau}

\usepackage{scrextend}

\usepackage{hyperref}

\theoremstyle{plain}
\theoremstyle{definition}
\newtheorem{theorem}{Theorem}[section]

\newtheorem{lemma}[theorem]{Lemma}

\newtheorem{definition}[theorem]{Definition}

\newtheorem{example}[theorem]{Example}
\newtheorem{proposition}[theorem]{Proposition}
\newtheorem{corollary}[theorem]{Corollary}

\DeclareMathAlphabet{\mathpzc}{OT1}{pzc}{m}{it}

\newcommand{\symm}{\mathfrak{S}}

\newcommand{\ol}[1]{\mbox{$\overline{#1}$}}
\usepackage{mathabx}
\renewcommand{\star}{\mbox{$\Asterisk$}}
\newcommand{\ncp}{\textsf{L-NC}}
\newcommand{\ncpn}{\textsf{L-NCN}}
\newcommand{\SF}{\textsf{SF}}
\newcommand{\Shift}{\textsf{Shift}}
\newcommand{\Star}{\textsf{Star}}
\newcommand{\NC}{\text{NC}}

\newcommand*\circled[1]{\tikz[baseline=(char.base)]{
            \node[shape=circle,draw,inner sep=2pt] (char) {#1};}}

\begin{document}

\title{Star factorizations and noncrossing partitions}

\date{}

\author{Bridget Eileen Tenner}
\address{Department of Mathematical Sciences, DePaul University, Chicago, IL, USA}
\email{bridget@math.depaul.edu}
\thanks{Research partially supported by Simons Foundation Collaboration Grant for Mathematicians 277603 and by a University Research Council Competitive Research Leave from DePaul University.}



\begin{abstract}
We develop the relationship between minimal transitive star factorizations and noncrossing partitions. This gives a new combinatorial proof of a result by Irving and Rattan, and a specialization of a result of Kreweras. It also arises in a poset on the symmetric group whose definition is motivated by the Subword Property of the Bruhat order.
\end{abstract}

\maketitle

\section{Introduction}

Let $\symm_n$ be the symmetric group on $[n] := \{1,\ldots, n\}$. For any $k \in [n]$, called the \emph{pivot}, the group $\symm_n$ is generated by the transpositions
$$\star_{n;k} := \big\{ (k \ i) : i \in [n] \setminus \{k\}\big\},$$
which we multiply from right to left. For $\pi \in \symm_n$, a decomposition $\pi = g_1 \cdots g_r$ for $g_i \in \star_{n;k}$ is a \emph{star factorization} of $\pi$. Following terminology of Goulden and Jackson \cite{goulden jackson} and Irving and Rattan \cite{irving rattan}, we say that this star factorization is \emph{transitive} if $\{g_1,\ldots,g_r\} = \star_{n;k}$. A \emph{minimal} transitive star factorization refers to minimality of the length $r$. If $\pi \in \symm_n$ has $m$ disjoint cycles, then a minimal transitive star factorization of $\pi$ has length $n + m - 2$ \cite{irving rattan, pak}. Let
$$\star_k(\pi)$$
denote the minimal transitive star factorizations of $\pi$ with pivot $k$. Throughout this work, all ``star factorizations'' are assumed to be minimal transitive star factorizations.

Irving and Rattan showed in \cite{irving rattan} that the number of star factorizations of a permutation $\pi \in \symm_n$ depended only on the cycle type of $\pi$: if $\pi$ 
has cycles of lengths $\ell_1$, \ldots, $\ell_m$, then
\begin{equation}\label{eqn:irving rattan}
|\star_k(\pi)| = \frac{(n+m-2)!}{n!} \cdot \ell_1\cdots \ell_m.
\end{equation}
(When $m>1$, we can write this as $(n+m-2)_{m-2}\ell_1\cdots \ell_m$, where $(a)_{b}$ is the falling factorial.) Perhaps most striking about this result is its independence of the pivot $k$. In \cite{tenner star}, we gave a combinatorial proof of that independence. Goulden and Jackson extended this independence to non-minimal factorizations in \cite{goulden jackson 2009}.

In this note, we give a new combinatorial proof of Equation~\eqref{eqn:irving rattan}, based on an interpretation of star factorizations in terms of noncrossing partitions. This is indicative of a close relationship between these two classes of objects, which can be used to prove/recover several interesting results. First, using the enumeration given by \cite{irving rattan} in Equation~\eqref{eqn:irving rattan}, we recover a result of Kreweras about enumerating noncrossing partitions with specified (labeled) part sizes \cite{kreweras}. This enumeration turns out to be independent of the part sizes themselves, and depends only on how many parts there are. Second, we show that in a Bruhat-style poset defined in terms of the minimal transitive star factorizations of elements of $\symm_n$, the poset's intervals are built out of noncrossing partition posets. Once again, the pivot value is immaterial, and the posets defined in this way for two different pivots are equal.

This interplay between star factorizations and noncrossing partitions not only highlights their connection, but also a notion of ``independent of $\ldots$'' that certain features of each may share.

Section~\ref{sec:star factorizations as ncp} reviews the main results of Irving and Rattan. These results are used to establish the relationship between star factorizations and labeled noncrossing partitions in Proposition~\ref{prop:valid noncrossing} and Theorem~\ref{thm:sf is a bijection}. This yields that the number of labeled noncrossing partitions of $X$ into $m$ parts of specified sizes is the falling factorial $(X)_{m-1}$ in Corollary~\ref{cor:counting ncp}. Section~\ref{sec:pivot independence} revisits the enumeration of \cite{irving rattan} and provides a new combinatorial proof of that result in Corollary~\ref{cor:irving rattan bijection}. In Section~\ref{sec:poset}, we define a poset on $\symm_n$ using star factorizations, motivated by the Subword Property of the Bruhat order. This is equivalent to a poset that can be defined in terms of permutation cycles (Theorem~\ref{thm:poset is cycles}), and is thus yet another example of pivot independence arising in a property defined by star factorizations (Corollary~\ref{cor:poset pivot independence}). The relationship between this poset and noncrossing partitions is discussed in Lemma~\ref{lem:single sliced cycle} and Corollary~\ref{cor:intervals are products of NC}, and further elucidated in examples and discussions in that section. We conclude with two open questions in Section~\ref{sec:future}.

\section{Star factorizations as labeled noncrossing partitions}\label{sec:star factorizations as ncp}

Before envisioning star factorizations as labeled noncrossing partitions, we review two results of Irving and Rattan.

Fix positive integers $n \ge k$ and $\pi \in \symm_n$. Suppose that $\pi$ consists of $m$ disjoint cycles, $C_1, \ldots, C_m$, where $C_i$ has length $\ell_i$. Unless specified otherwise, cycles are indexed in order of their minimal elements. Let $p$ be the index of the cycle containing the pivot $k$.

\begin{lemma}[\cite{irving rattan}]\label{lem:how cycles get factored}
Let $\delta$ be a star factorization of $\pi$.
\begin{enumerate}
\item[(a)] If $C_i = (a_1 \ a_2 \ \cdots \ a_t)$, where $i \neq p$, then some transposition $(k \ a_j)$ appears exactly twice in $\delta$ and all $(k \ a_h)$ with $h \neq j$ appear exactly once in $\delta$. These appear as $(k \ a_j) (k \ a_{j-1}) \cdots (k \ a_1) (k \ a_t) \cdots (k \ a_{j+1}) (k \ a_j)$, from left to right in $\delta$.
\item[(b)] If $C_p = (k \ b_2 \ \cdots \ b_t)$, then each transposition $(k \ b_h)$ appears exactly once in $\delta$. These appear as $(k \ b_t) \cdots (k \ b_2)$, from left to right in $\delta$.
\end{enumerate}
\end{lemma}

We can turn $\delta \in \star_k(\pi)$ into a word
$$\omega(\delta) \in [m]^{n+m-2}$$
recording the index of the cycle containing $i$ for each factor $(k \ i)$.

\begin{example}\label{ex:star factorization}
Let $\pi = (13)(285)(4)(67)$ and $k = 6$, with
$$\delta = (6 \ 8) (6 \ 1) (6 \ 3) (6 \ 1) (6 \ 2) (6 \ 5) (6 \ 8) (6 \ 7) (6 \ 4) (6 \ 4).$$
Then $p = 4$ because $6$ appears in the fourth cycle, and
$$\omega(\delta) = 2111222433.$$
\end{example}

Irving and Rattan characterize the possible words $\omega(\delta)$ that are formed by $\delta \in \star_k(\pi)$.

\begin{definition}\label{defn:valid words}
A word $\omega \in [m]^{n+m-2}$ is \emph{valid} if
\begin{itemize}
\item the symbol $p$ appears $\ell_p - 1$ times,
\item for all $j \in [m] \setminus \{p\}$, the symbol $j$ appears $\ell_j+1$ times,
\item for $i \neq j$, there is no subword $ijij$ in $\omega(\delta)$, and
\item for $i \neq p$, there is no subword $ipi$ in $\omega(\delta)$.
\end{itemize}
\end{definition}

\begin{lemma}[\cite{irving rattan}]\label{lem:how cycles interact}
A word $\omega \in [m]^{n+m-2}$ is valid if and only if $\omega = \omega(\delta)$ for some $\delta \in \star_k(\pi)$.
\end{lemma}

The third requirement in Definition~\ref{defn:valid words} suggests a kind of noncrossing phenomenon, but the fourth requirement is not quite on target. However, if we add a copy of $p$ to either end of a valid word $\omega$, then we can indeed rephrase this as a noncrossing problem.

\begin{definition}
Given a valid word $\omega$ for $\pi$, let $\ol{\omega}$ be the necklace obtained by inserting $p$ between the first and last letters of $\omega$.
\end{definition}

For consistency, we will read necklaces in counterclockwise order.

\begin{example}\label{ex:valid noncrossing}
The word $\omega = 2111222433$, valid for $(13)(285)(4)(67)$ with pivot $k = 6$, corresponds to the necklace $\ol{\omega}$ displayed in Figure~\ref{fig:necklace}, with the inserted copy of $4$ circled.
\begin{figure}[htbp]
$$\begin{tikzpicture}
\node[draw=none,minimum size=4.7cm,regular polygon,regular polygon sides=11] (b) {};
\node[draw=black,style=dotted,minimum size=4cm,regular polygon,regular polygon sides=11] (a) {};
\foreach \x in {1,...,11} {\fill (a.corner \x) circle (2pt);}
\draw (b.corner 1) node {$1$};
\draw (b.corner 2) node {$1$};
\draw (b.corner 3) node {$1$};
\draw (b.corner 4) node {$2$};
\draw (b.corner 5) node {$2$};
\draw (b.corner 6) node {$2$};
\draw (b.corner 7) node {$4$};
\draw (b.corner 8) node {$3$};
\draw (b.corner 9) node {$3$};
\draw (b.corner 10) node {$\circled{4}$};
\draw (b.corner 11) node {$2$};
\end{tikzpicture}$$

\caption{The necklace $\ol{2111222433}$, in which $4$ (circled in the figure) has been inserted between the first and last letters.}\label{fig:necklace}
\end{figure}
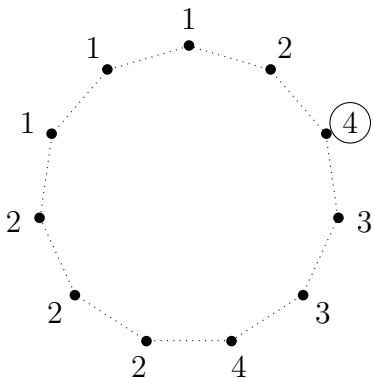
\end{example}

These necklaces allow us to recharacterize valid words as noncrossing partitions.

\begin{definition}
Fix a sequence of positive integers $x = (x_1,\ldots, x_t)$. A \emph{labeled noncrossing partition} of \emph{type} $x$ is a noncrossing partition of $[x_1 + \cdots + x_t]$ in which the parts have sizes $(x_1,\ldots, x_t)$, and the part of size $x_i$ (equivalently, each element in that part) is labeled $i$. Let 
$$\ncp(x)$$
be the set of labeled noncrossing partitions of type $x$. The rotation classes of these objects are the \emph{labeled noncrossing necklaces} of \emph{type} $x$, denoted
$$\ncpn(x) := \ncp(x)/\text{rotation}.$$
\end{definition}

\begin{example}
As illustrated in Figure~\ref{fig:ncp example}, $|\ncp(2,2)| = 4$ and $|\ncpn(2,2)| = 1$.
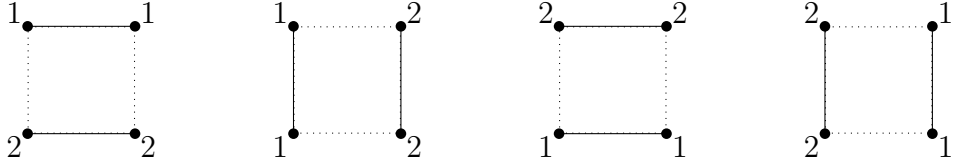
\begin{figure}[htbp]
$$\begin{tikzpicture}
\node[draw=none,minimum size=2.5cm,regular polygon,regular polygon sides=4] (b) {};
\node[draw=black,style=dotted,minimum size=2cm,regular polygon,regular polygon sides=4] (a) {};
\foreach \x in {1,...,4} {\fill (a.corner \x) circle (2pt);}
\draw (a.corner 1) -- (a.corner 2);
\draw (a.corner 3) -- (a.corner 4);
\draw (b.corner 1) node {$1$};
\draw (b.corner 2) node {$1$};
\draw (b.corner 3) node {$2$};
\draw (b.corner 4) node {$2$};
\end{tikzpicture}
\hspace{.5in}
\begin{tikzpicture}
\node[draw=none,minimum size=2.5cm,regular polygon,regular polygon sides=4] (b) {};
\node[draw=black,style=dotted,minimum size=2cm,regular polygon,regular polygon sides=4] (a) {};
\foreach \x in {1,...,4} {\fill (a.corner \x) circle (2pt);}
\draw (a.corner 2) -- (a.corner 3);
\draw (a.corner 1) -- (a.corner 4);
\draw (b.corner 1) node {$2$};
\draw (b.corner 2) node {$1$};
\draw (b.corner 3) node {$1$};
\draw (b.corner 4) node {$2$};
\end{tikzpicture}
\hspace{.5in}
\begin{tikzpicture}
\node[draw=none,minimum size=2.5cm,regular polygon,regular polygon sides=4] (b) {};
\node[draw=black,style=dotted,minimum size=2cm,regular polygon,regular polygon sides=4] (a) {};
\foreach \x in {1,...,4} {\fill (a.corner \x) circle (2pt);}
\draw (a.corner 1) -- (a.corner 2);
\draw (a.corner 3) -- (a.corner 4);
\draw (b.corner 1) node {$2$};
\draw (b.corner 2) node {$2$};
\draw (b.corner 3) node {$1$};
\draw (b.corner 4) node {$1$};
\end{tikzpicture}
\hspace{.5in}
\begin{tikzpicture}
\node[draw=none,minimum size=2.5cm,regular polygon,regular polygon sides=4] (b) {};
\node[draw=black,style=dotted,minimum size=2cm,regular polygon,regular polygon sides=4] (a) {};
\foreach \x in {1,...,4} {\fill (a.corner \x) circle (2pt);}
\draw (a.corner 2) -- (a.corner 3);
\draw (a.corner 1) -- (a.corner 4);
\draw (b.corner 1) node {$1$};
\draw (b.corner 2) node {$2$};
\draw (b.corner 3) node {$2$};
\draw (b.corner 4) node {$1$};
\end{tikzpicture}$$
\caption{The four labeled noncrossing partitions of type $(2,2)$, which all are rotationally equivalent.}\label{fig:ncp example}
\end{figure}
\end{example}

Throughout this section, we will refer to the sequence $\ell' = (\ell'_1,\ldots, \ell'_m)$, where
\begin{equation}\label{eqn:ell'}
\ell'_i := \begin{cases} \ell_i + 1 & \text{if } i \in [m]\setminus \{p\}, \text{ and}\\
\ell_p & \text{if } i = p. \end{cases}
\end{equation}

\begin{proposition}\label{prop:valid noncrossing}
Consider a necklace $\alpha$ with $\ell'_i$ copies of $i$, for all $i \in [m]$. Then $\alpha \in \ncpn(\ell')$ if and only if $\alpha = \ol{\omega}$ for some valid word $\omega \in [m]^{n+m-2}$.
\end{proposition}

\begin{proof}
If $\alpha \in \ncpn(\ell')$, then removing any copy of $p$ and reading counterclockwise will produce a string $\omega$ that satisfies all requirements of Definition~\ref{defn:valid words}. Thus $\alpha = \ol{\omega}$, with that chosen copy of $p$ being the inserted one. Now suppose that $\alpha$ has a crossing, and consider all $\ell'_p$ words obtained by removing one copy of $p$ from $\alpha$ and reading the remaining letters in counterclockwise order. If the crossing in $\alpha$ involves values other than $p$, then that crossing will also appear in the resulting word, violating the third requirement for validity. On the other hand, if that crossing involves $p$ and some letter $i$, then there is a substring $ipi$ appearing in all $\ell'_p$ words obtained in this way, violating the fourth requirement for validity.
\end{proof}

Proposition~\ref{prop:valid noncrossing} is illustrated in Figure~\ref{fig:valid noncrossing}, continuing Example~\ref{ex:valid noncrossing}.

\begin{figure}[htbp]
$$\begin{tikzpicture}
\node[draw=none,minimum size=4.7cm,regular polygon,regular polygon sides=11] (b) {};
\node[draw=black,style=dotted,minimum size=4cm,regular polygon,regular polygon sides=11] (a) {};
\foreach \x in {1,...,11} {\fill (a.corner \x) circle (2pt);}
\draw[thick] (a.corner 11) -- (a.corner 4) -- (a.corner 5) -- (a.corner 6) -- (a.corner 11);
\draw[thick] (a.corner 1) -- (a.corner 2) -- (a.corner 3) -- (a.corner 1);
\draw[thick] (a.corner 7) -- (a.corner 10);
\draw[thick] (a.corner 8) -- (a.corner 9);
\draw (b.corner 11) node {$2$};
\draw (b.corner 1) node {$1$};
\draw (b.corner 2) node {$1$};
\draw (b.corner 3) node {$1$};
\draw (b.corner 4) node {$2$};
\draw (b.corner 5) node {$2$};
\draw (b.corner 6) node {$2$};
\draw (b.corner 7) node {$4$};
\draw (b.corner 8) node {$3$};
\draw (b.corner 9) node {$3$};
\draw (b.corner 10) node {$\circled{4}$};
\end{tikzpicture}$$
\caption{A labeled noncrossing necklace of type $(3,4,2,2)$, marked to match Figure~\ref{fig:necklace}.}\label{fig:valid noncrossing}
\end{figure}
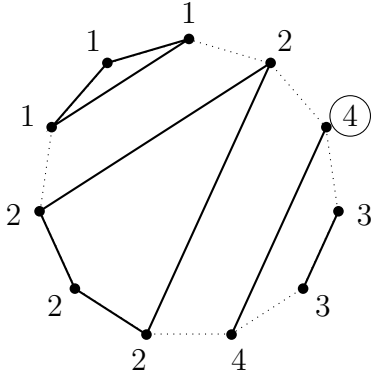

\begin{definition}\label{defn:map sf}
Define the map
\begin{align*}
\textsf{SF}_k : \ncpn(\ell') \times [\ell_1] \times \cdots \times [\ell_m] &\rightarrow \star_k(\pi)\\
(\alpha,d_1,\ldots,d_m) &\mapsto \delta 
\end{align*}
as follows. First, fix a presentation of $\alpha$ and remove the $d_p$th appearance of $p$ when reading the necklace in counterclockwise order from the topmost letter. Let $\omega$ be the (valid) word formed by reading the remainder of $\alpha$ in counterclockwise order from this removal (hence $\alpha = \ol{\omega}$). We know how to interpret each $p$ in this word by Lemma~\ref{lem:how cycles get factored}(a). For each $j \in [m] \setminus \{p\}$ let $(k \ c_j)$ be the factor that will appear twice in the $\delta$, where $c_j$ is the $d_j$th smallest value in cycle $C_j$. Thus, by Lemma~\ref{lem:how cycles get factored}(b), we have completely defined $\delta$.
\end{definition}

\begin{example}
Let $\pi = (13)(285)(4)(67)$ and $k = 6$, so $\ell' = (3,4,2,2)$. Let $\alpha$ be the necklace depicted in Figure~\ref{fig:valid noncrossing}. Let $(d_1,d_2,d_3,d_4) = (1,3,1,2)$. This $d_4$ removes the circled copy of $4$ in Figure~\ref{fig:valid noncrossing}, and $\omega = 2111222433$. This $d_1$ means that the $1$s in $\omega$ should be replaced by $(6 \ 1)(6 \ 3)(6 \ 1)$, while $d_2$ means that $(6 \ 8)$ should appear twice in $\delta$, and $d_3$ indicates that $(6 \ 4)$ should appear twice in $\delta$:
$$\delta = (6 \ 8) (6 \ 1) (6 \ 3) (6 \ 1) (6 \ 2) (6 \ 5) (6 \ 8) (6 \ 7) (6 \ 4) (6 \ 4) \in \star_6((13)(285)(4)(67)),$$
as in Example~\ref{ex:star factorization}.
If, instead, we had used $(d_1,d_2,d_3,d_4) = (2,3,1,1)$, then the resulting star factorization would have been a cyclic rotation of this word $\delta$:
$$(6 \ 4) (6 \ 4) (6 \ 7) (6 \ 8) (6 \ 3) (6 \ 1) (6 \ 3) (6 \ 2) (6 \ 5) (6 \ 8) \in \star_6((13)(285)(4)(67)).$$
\end{example}

\begin{theorem}\label{thm:sf is a bijection}
The map $\SF_k$ is a bijection.
\end{theorem}

\begin{proof}
Suppose that $\SF_k(\alpha,d_1,\ldots,d_m) = \SF_k(\beta,e_1,\ldots,e_m)$. The valid words defined by $(\alpha, d_p)$ and $(\beta,e_p)$ must be the same in order to yield the same star factorization. Thus the necklaces $\alpha$ and $\beta$ differ only by rotation and hence are the same. Suppose that $d_p$ and $e_p$ differ. Then, reading from the $d_p$th copy of $p$, the letters of the necklace are $p \omega_1\omega_2\cdots\omega_t p \omega_1 \omega_2\cdots \omega_t$. If some $\omega_i \neq p$ then $\alpha$ would have a crossing, which is a contradiction. On the other hand, if all $\omega_i = p$, then there is only one necklace and so only one cycle in $\pi$ and only one factorization. In order to have the correct factors appear twice in the factorization output by $\SF_k$, the values $d_i$ and $e_i$ must be equal for all $i \in [m] \setminus \{p\}$. Thus $\SF_k$ is injective.

To show that $\SF_k$ is surjective, we can take any $\delta = \star_k(\pi)$ and define $d_i$ for all $i \in [m] \setminus \{p\}$ as required by Definition~\ref{defn:map sf}. That is, if $(k \ c_j)$ appears twice in $\delta$, then define $d_j$ so that $c_j$ is the $d_j$th smallest value in the cycle $C_j$. Form $\ol{\omega}$ from the word $\omega(\delta)$, oriented so that, say, the first $1$ in $\omega(\delta)$ is the top bead in $\ol{\omega}$, and let $d_p$ indicate the letter $p$ that was added to form this necklace. Thus $\SF_k(\ol{\omega},d_1,\ldots,d_m) = \delta$.
\end{proof}

Theorem~\ref{thm:sf is a bijection} and Equation~\eqref{eqn:irving rattan} yield a central relationship of this paper; namely, that the number of non-crossing necklaces of specified type depends only on the number of parts in that specification.

\begin{corollary}\label{cor:necklace count independent}
For $x = (x_1,\ldots,x_m)$ with at most one $x_i$ equal to $1$,
$$|\ncpn(x)| = (\textstyle\sum x_i - 1)_{m-2}.$$
\end{corollary}

We can use Theorem~\ref{thm:sf is a bijection} and Corollary~\ref{cor:necklace count independent} to enumerate labeled noncrossing partitions with specified part sizes. This can also be shown using a result of Kreweras \cite{kreweras}. It is surprising that this depends only on the number of specified part sizes, not on the sizes themselves.

\begin{corollary}\label{cor:counting ncp}
The number of labeled noncrossing partitions of $X$ into $m$ parts of specified sizes is the falling factorial $(X)_{m-1}$.
\end{corollary}

\begin{proof}
If $m = 1$ then there is clearly $1 = (X)_{1-1}$ labeled noncrossing partition of $X$.

Now assume that $m > 1$, and write $|\cdot|$ for the sum of the terms in a sequence. Thus, for any sequence $x = (x_1,\ldots,x_m)$ of positive integers with at most one $x_i = 1$, we have Corollary~\ref{cor:necklace count independent}, which says that $|\ncpn(x)| = (|x|-1)_{m-2}$. We can expand this to allow more $1$s by inducting on $m$. If $m = 2$, then certainly $|\ncpn(x)| = 1 = (|x|-1)_{2-2}$. Now suppose that the result holds for $m \ge 2$ and consider an $(m+1)$-tuple $x^+$ having at least two $1$s. Without loss of generality, assume that $x^+ = (x_1,\ldots, x_m, 1)$. Set $x := (x_1,\ldots,x_m)$. Then each labeled noncrossing necklace of type $x^+$ can be obtained by inserting the label $m+1$ anywhere among the $|x|$ letters in any element of $\ncpn(x)$. Thus
$$|\ncpn(x^+)| = |x|\cdot|\ncpn(x)| = |x| \cdot (|x|-1)_{m-2} = (|x|)_{m-1} = (|x^+|-1)_{(m+1)-2},$$
as desired.

It remains now to convert noncrossing necklaces to noncrossing partitions. If $m = 1$, then $|\ncp(x)| = 1 = |x|_{1-1}$. On the other hand, if $m > 1$, then the $|x|$ rotations of a necklace of type $x$ are all distinct in $\ncp(x)$. Therefore
\begin{align*}
|\ncp(x)| &= |x| \cdot |\ncpn(x)|\\
&= |x| \cdot (|x|-1)_{m-2}\\
&= (|x|)_{m-1}.
\end{align*}
\end{proof}

\section{Pivot independence in star factorizations}\label{sec:pivot independence}

Theorem~\ref{thm:sf is a bijection} gives a correspondence between labeled noncrossing necklaces and star factorizations. We can use those partitions as an intermediary between elements of $\star_k(\pi)$ and elements of $\star_{k'}(\pi)$, giving a new combinatorial proof of Equation~\eqref{eqn:irving rattan}.

Fix positive integers $n \ge k > k'$ and a permutation $\pi \in \symm_n$. As before, suppose that $\pi$ consists of $m$ disjoint cycles, $C_1,\ldots,C_m$, where $C_i$ has length $\ell_i$. Let $p$ be the index of the cycle containing the pivot $k$ and let $p'$ be the index of the cycle containing the pivot $k'$. Define sequences $a = (a_1,\ldots,a_m)$ and $b = (b_1,\ldots,b_m)$ as
$$a_i := \begin{cases} \ell_i+1 & \text{if } i \in [m] \setminus \{p\}\text{, and}\\ \ell_p & \text{if }i = p,\end{cases}
\hspace{.25in} \text{and} \hspace{.25in}
b_i := \begin{cases} \ell_i+1 & \text{if } i \in [m] \setminus \{p'\}\text{, and}\\ \ell_{p'} & \text{if }i = p'.\end{cases}$$
Fix $\delta \in \star_k(\pi)$. Our goal will be to define a corresponding $\delta' \in \star_{k'}(\pi)$. We will do this using $\SF_k$ from the previous section, and the map $\Shift$ defined below.

\begin{definition}\label{defn:shift}
Define the map
\begin{align*}
\Shift : \ncpn(a) \times [\ell_1]\times\cdots\times[\ell_m] &\rightarrow \ncpn(b) \times [\ell_1]\times\cdots\times[\ell_m]\\
(\alpha,d_1,\ldots,d_m) &\mapsto (\beta,d_1,\ldots,d_m)
\end{align*}
as follows. If $p = p'$, then set $\beta := \alpha$. Otherwise, first fix a presentation of $\alpha$ and locate the $d_p$th copy of $p$ when reading the necklace in counterclockwise order from the topmost letter. Counterclockwise from there, locate the first copy of $p'$, reading a string of the form $ps_1\cdots s_t$, where $s_t = p'$. Let $h \in [1,t]$ be maximal such that the necklace obtained by replacing the string $ps_1\cdots s_{t-1} p'$ by $ps_1\cdots s_{h-1} p s_h \cdots s_{t-1}$ would not have any crossings. Let $\beta \in \ncpn(b)$ be the necklace obtained by this replacement.
\end{definition}

Note that the value of $d_p$ determines which copy of $p$ to use for orientation in $\Shift$.

We know that the $h$ described in Definition~\ref{defn:shift} exists because $s_1$ has the desired property.

\begin{example}\label{ex:beta}
Let $n = 8$, $k = 6$, and $k' = 3$, and let $\pi = (13)(285)(4)(67)$. Then $p = 4$ and $p' = 1$. Let $\alpha$ be the necklace depicted in Figure~\ref{fig:valid noncrossing}. Then $\Shift(\alpha,1,3,1,2) = (\beta,1,3,1,2)$, where $\beta$ is the necklace depicted in Figure~\ref{fig:beta}, with the $d_1$st counterclockwise-from-top appearance of $p'$ circled.
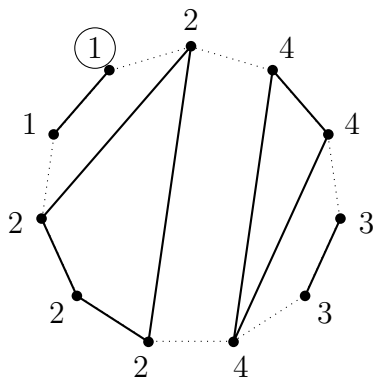
\begin{figure}[htbp]
$$
\begin{tikzpicture}
\node[draw=none,minimum size=4.7cm,regular polygon,regular polygon sides=11] (b) {};
\node[draw=black,style=dotted,minimum size=4cm,regular polygon,regular polygon sides=11] (a) {};
\foreach \x in {1,...,11} {\fill (a.corner \x) circle (2pt);}
\draw[thick] (a.corner 1) -- (a.corner 4) -- (a.corner 5) -- (a.corner 6) -- (a.corner 1);
\draw[thick] (a.corner 2) -- (a.corner 3);
\draw[thick] (a.corner 7) -- (a.corner 10) -- (a.corner 11) -- (a.corner 7);
\draw[thick] (a.corner 8) -- (a.corner 9);
\draw (b.corner 11) node {$4$};
\draw (b.corner 1) node {$2$};
\draw (b.corner 2) node {$\circled{1}$};
\draw (b.corner 3) node {$1$};
\draw (b.corner 4) node {$2$};
\draw (b.corner 5) node {$2$};
\draw (b.corner 6) node {$2$};
\draw (b.corner 7) node {$4$};
\draw (b.corner 8) node {$3$};
\draw (b.corner 9) node {$3$};
\draw (b.corner 10) node {$4$};
\end{tikzpicture}$$
\caption{The image of $\Shift(\alpha,1,3,1,2)$, where $\alpha$ appeared in Figure~\ref{fig:valid noncrossing}.}\label{fig:beta}
\end{figure}
\end{example}

\begin{proposition}
The map $\Shift$ is invertible.
\end{proposition}

\begin{proof}
Let $\Shift' : \ncpn(b) \times [\ell_1]\times\cdots\times[\ell_m] \rightarrow \ncpn(a) \times [\ell_1]\times\cdots\times[\ell_m]$ be defined exactly as $\Shift$, except for replacing ``counterclockwise from there'' by ``clockwise from there.'' Then $\Shift' \circ \Shift$ and $\Shift \circ \Shift'$ are the identity maps on their respective spaces.
\end{proof}

\begin{corollary}\label{cor:irving rattan bijection}
The map
$$\SF_{k'} \circ \Shift \circ \SF_k^{-1} : \star_k(\pi) \rightarrow \star_{k'}(\pi)$$
is a bijection.
\end{corollary}

Thus we have a bijection between star factorizations of $\pi$ with pivot $k$, and star factorizations of $\pi$ with pivot $k'$, giving a new combinatorial proof of Equation~\eqref{eqn:irving rattan}. Note that the bijection described in Corollary~\ref{cor:irving rattan bijection} is not the same as the bijection presented in \cite[Section 6]{tenner star}.

\begin{example}
Continuing Example~\ref{ex:beta}, the composition $\SF_3 \circ \Shift \circ \SF_6^{-1}$ would pair the star factorization
$$\delta = (6 \ 8) (6 \ 1) (6 \ 3) (6 \ 1) (6 \ 2) (6 \ 5) (6 \ 8) (6 \ 7) (6 \ 4) (6 \ 4) \in \star_6((13)(285)(4)(67))$$
with the star factorization
$$\delta' = (3 \ 1) (3 \ 8) (3 \ 2) (3 \ 5) (3 \ 7) (3 \ 4) (3 \ 4) (3 \ 6) (3 \ 7) (3 \ 8) \in \star_3((13)(285)(4)(67)).$$
They would be paired via $\delta \mapsto (\alpha,1,3,1,2) \mapsto (\beta,1,3,1,2) \mapsto \delta'$.
\end{example}

\section{A Bruhat-style poset on star factorizations}\label{sec:poset}

The Subword Property of the Bruhat order says that we can fix a reduced decomposition of a permutation $\pi$, and find all $\sigma \le \pi$ in the Bruhat order by deleting letters from that reduced decomposition \cite{bjorner brenti}. This suggests a new poset structure to $\symm_n$ in terms of minimal transitive star factorizations.

\begin{definition}\label{defn:bruhat-like order}
Fix positive integers $n \ge k$. For $\sigma, \pi \in \symm_n$, say that $\sigma \preceq_k \pi$ if there exists $\gamma \in \star_k(\sigma)$ that is a subword (not necessarily contiguous) of $\delta \in \star_k(\pi)$. Let $\Star_k(n)$ be the poset defined by $\preceq_k$ on $\symm_n$.
\end{definition}

Note that Definition~\ref{defn:bruhat-like order} is similar to the Subword Property of the Bruhat order, but not identical to it because it does not require the same $\delta \in \star_k(\pi)$ to work for all $\sigma\preceq_k\pi$. The poset defined by $\preceq_1$ on $\symm_3$ appears in Figure~\ref{fig:s3}. 

\begin{figure}[htbp]
$$\begin{tikzpicture}
\foreach \x in {-2,0,2} {\draw (-1,0) -- (\x,2) -- (1,0); \draw (0,4) -- (\x,2);}
\node at (-1,0) [rectangle,draw,fill=white] {$(132)$};
\node at (1,0) [rectangle,draw,fill=white] {$(123)$};
\node at (-2,2) [rectangle,draw,fill=white] {$(13)(2)$};
\node at (0,2) [rectangle,draw,fill=white] {$(1)(23)$};
\node at (2,2) [rectangle,draw,fill=white] {$(12)(3)$};
\node at (0,4) [rectangle,draw,fill=white] {$(1)(2)(3)$};
\end{tikzpicture}$$
\caption{The poset $\Star_1(3)$.}\label{fig:s3}
\end{figure}
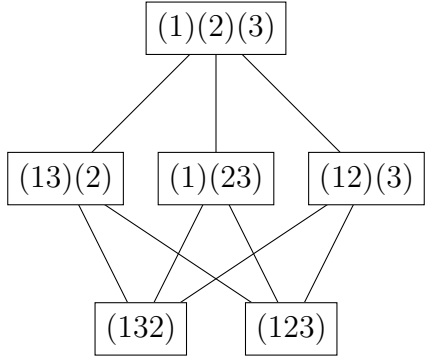

\begin{definition}\label{defn:excerpt}
Let $C = (c_1 \ \cdots \ c_t)$ be a cycle. An \emph{excerpt} $D$ of $C$ is a cycle of the form $D = (c_i \ c_{i+1} \ \cdots \ c_{i+j})$ for some positive integers $i$ and $j$, where indices are taken modulo $t$. A nonempty excerpt $D$ is \emph{proper} if $j<t$. The cycle $C$ has been \emph{sliced} if it has been partitioned into proper excerpts. 
\end{definition}

Characterization of the poset $\Star_k(n)$ follows, essentially, from Lemmas~\ref{lem:how cycles get factored} and~\ref{lem:how cycles interact}.

\begin{theorem}\label{thm:poset is cycles}
$\sigma \preceq_k \pi$ if and only if every cycle of $\pi$ is an excerpt of a cycle of $\sigma$.
\end{theorem}

\begin{proof}
Suppose that $\sigma \preceq_k \pi$, with $\gamma$ and $\delta$ as in Definition~\ref{defn:bruhat-like order}. Because these are minimal \emph{transitive} star factorizations, any transposition that gets deleted from $\delta$ to form $\gamma$ must have been duplicated in $\delta$, by Lemma~\ref{lem:how cycles get factored}. Using the notation of that lemma, let $(k \ a_j)$ be a transposition deleted from $\delta$ to form $\gamma$. Suppose first that the factors from the $C_i$ containing $a_j$ had appeared between two factors from some other $C_q = (\cdots xy \cdots)$ in $\delta$, as
$$(k \ y) \cdots (k \ a_j) (k \ a_{j-1}) \cdots (k \ a_1)(k \ a_t) \cdots (k \ a_{j+1}) (k \ a_j) \cdots (k \ x),$$
and let $x$ and $y$ be the nearest neighbors with this property. Then deleting either copy of $(k \ a_j)$ will merge the two cycles, forming
$$(\cdots x \ a_j \ a_{j+1} \cdots a_t \ a_1 \cdots a_{j-1} \ y\cdots) \hspace{.25in} \text{or} \hspace{.25in}
(\cdots x \ a_{j+1} \cdots a_t \ a_1 \cdots a_{j-1} \ a_j \ y\cdots),$$
depending on which $(k \ a_j)$ factor was deleted. 
On the other hand, if the factors from this $C_i$ do not appear between two factors from any other $C_q$, then $C_i$ merges with the cycle $C_p$ containing the pivot value, as indicated by Lemma~\ref{lem:how cycles get factored}.

Now suppose that every cycle of $\pi$ is an excerpt of a cycle of $\sigma$. Orient the cycles in each permutation so that if the parentheses are deleted then the two resulting words are identical. (For a small example in $\symm_3$, we could write $\pi = (1)(32)$ and $\sigma = (132)$.) Construct $\gamma \in \star_k(\sigma)$ using Lemma~\ref{lem:how cycles get factored}, factoring each cycle from left to right. Do likewise to construct $\delta \in \star_k(\pi)$. This $\gamma$ is necessarily a subword of $\delta$, and so $\sigma \preceq_k \pi$.
\end{proof}

Reminiscent of the work of \cite{irving rattan} and the previous section, Theorem~\ref{thm:poset is cycles} gives another situation in which the pivot value itself does not matter.

\begin{corollary}\label{cor:poset pivot independence}
For all $k, k' \in [n]$, $\Star_k(n) = \Star_{k'}(n)$.
\end{corollary}

In other words, we can call this poset
$$\Star(n)$$
and not specify the pivot $k$. Accordingly, we can write $\preceq$ for $\preceq_k$. The ordering characterization given in Theorem~\ref{thm:poset is cycles} allows us to describe the structure of this poset in detail. For example, the identity $(1)(2) \cdots (n)$ is the unique maximal element, and there are $(n-1)!$ minimal elements: the $n$-cycles.

\begin{lemma}\label{lem:single sliced cycle}
Fix $\sigma \preceq \pi$ in $\Star(n)$. Suppose that exactly one cycle of $\sigma$ has been sliced to form $\pi$, and that it resulted in $d$ proper excerpts in $\pi$. Then $[\sigma,\pi]$ is isomorphic to the noncrossing partition lattice $\NC(d)$.
\end{lemma}

\begin{proof}
Let the disjoint cycles of $\sigma$ be $C_1, \ldots, C_m$, and suppose that $C_m$ has been sliced into proper excerpts $X_1, \ldots, X_d$ in $\pi$. By Theorem~\ref{thm:poset is cycles}, all elements of $[\sigma,\pi]$ have the form
$$C_1\cdots C_{m-1} Y_1\cdots Y_{d'},$$
where $C_m$ can be sliced into $Y_1\cdots Y_{d'}$, and the cycles $Y_1\cdots Y_{d'}$ can subsequently be sliced to produce $X_1\cdots X_d$. Therefore, in fact, we have
$$[\sigma,\pi]\cong [(12\cdots d),(1)(2)\cdots(d)] \subset \Star(d).$$

We will now show that $\mathcal{P} := [(12\cdots d),(1)(2)\cdots(d)]$ is isomorphic to the dual of $\NC(d)$. The poset $\NC(d)$ is self-dual, which will complete the proof.

The elements of $\mathcal{P}$ are in bijection with elements of $\NC(d)$, where the cycles of $Z \in \mathcal{P}$ give the corresponding partition $\phi(Z)$ of $[d]$. The correspondence between order relations follows from Theorem~\ref{thm:poset is cycles}. That is, if $Z \lessdot V$ in $\mathcal{P}$, then $V$ and $Z$ have all the same cycles, except for one cycle $Z_i$ in $Z$ that has been sliced into $V_1$ and $V_2$ in $V$. By Definition~\ref{defn:excerpt}, this amounts to partitioning the part of $\phi(Z)$ corresponding to $Z_i$ into two noncrossing pieces as defined by $V_1$ and $V_2$; in other words, $\phi(V)$ is covered by $\phi(Z)$ in $\NC(d)$. Similarly, noncrossing partitions $A$ covering $B$ in $\NC(d)$ correspond to $\phi^{-1}(A) \lessdot \phi^{-1}(B)$ in $\mathcal{P}$.

Therefore $[\sigma,\pi] \cong [(12\cdots d),(1)(2)\cdots(d)] \cong (\NC(d))^* \cong \NC(d)$.
\end{proof}

In \cite{brady}, Brady considered a poset similar to the one we study here. More precisely, his poset is dual to our interval $[(12\cdots d),(1)(2)\cdots(d)]$. The novelty that we bring, then, is not the poset itself but rather that its order relation can be realized through star transpositions.

\begin{corollary}\label{cor:intervals are products of NC}
Suppose that $\sigma \preceq \pi$ in $\Star(n)$, that the cycles of $\sigma$ that get sliced are $C_1, \ldots, C_s$ (meaning that all other cycles of $\sigma$ appear identically in $\pi$), and that $C_i$ results in $d_i$ proper excerpts in $\pi$. Then $[\sigma,\pi] \cong \NC(d_1) \times \cdots \times \NC(d_s)$.
\end{corollary}

\begin{proof}
This follows from Lemma~\ref{lem:single sliced cycle} and the fact that disjoint cycles act independently.
\end{proof}

\begin{example}
The interval $[(12345)(678),(15)(23)(4)(67)(8)]$ is isomorphic to $\NC(3)\times \NC(2)$, and is depicted in Figure~\ref{fig:interval example}.
\begin{figure}[htbp]
$$\begin{tikzpicture}
\foreach \x in {-6,-2,2,6} {\draw (0,0) -- (\x,2); \draw (0,6) -- (\x,4);}
\foreach \x in {-6,-2,2} {\draw (-6,4) -- (\x,2);}
\foreach \x in {-2,2,6} {\draw (6,2) -- (\x,4);}
\draw (-6,2) -- (-2,4);
\draw (-2,2) -- (2,4);
\draw (2,2) -- (6,4);
\node at (0,0) [rectangle,draw,fill=white] {$(12345)(678)$};
\node at (-6,2) [rectangle,draw,fill=white] {$(1235)(4)(678)$};
\node at (-2,2) [rectangle,draw,fill=white] {$(145)(23)(678)$};
\node at (2,2) [rectangle,draw,fill=white] {$(15)(234)(678)$};
\node at (6,2) [rectangle,draw,fill=white] {$(12345)(67)(8)$};
\node at (-6,4) [rectangle,draw,fill=white] {$(15)(23)(4)(678)$};
\node at (-2,4) [rectangle,draw,fill=white] {$(1235)(4)(67)(8)$};
\node at (2,4) [rectangle,draw,fill=white] {$(145)(23)(67)(8)$};
\node at (6,4) [rectangle,draw,fill=white] {$(15)(234)(67)(8)$};
\node at (0,6) [rectangle,draw,fill=white] {$(15)(23)(4)(67)(8)$};
\end{tikzpicture}$$
\caption{The interval $[(12345)(678),(15)(23)(4)(67)(8)]\subset\Star(8)$.}\label{fig:interval example}
\end{figure}
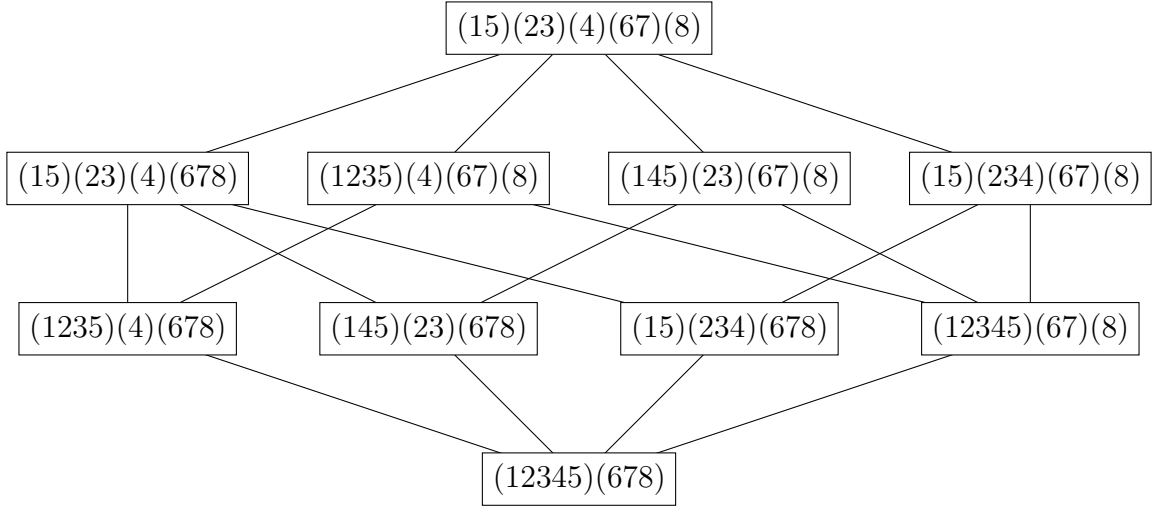
\end{example}

Suppose that $\pi \in \symm_n$ consists of $m$ disjoint cycles, $C_1, \ldots, C_m$, where $C_i$ has length $\ell_i$. Additional conclusions we can draw from Theorem~\ref{thm:poset is cycles} include:
\begin{itemize}
\item $\Star(n)$ is not \emph{bounded} because it has multiple ($(n-1)!$) minimal elements;
\item $\Star(n)$ is \emph{graded} because all maximal chains have $n$ elements;
\item as illustrated in Figure~\ref{fig:s3}, two elements of $\Star(n)$ may not have a greatest lower bound (nor, even, a common lower bound) and thus $\Star(n)$ is not a lattice; 
\item $\pi$ covers $\sum_{i<j}^m \ell_i \ell_j$ elements, formed by picking orientations of any two cycles and merging them in those orientations;
\item $\pi$ is covered by $\sum_i \binom{\ell_i}{2}$ elements, because a cycle of $\ell$ elements can be sliced into two proper excerpts by choosing two elements to serve as the first letters in each excerpt.
\end{itemize}

In previous work, we studied boolean elements in the Bruhat order \cite{claesson kitaev ragnarsson tenner, ragnarsson tenner 1, ragnarsson tenner 2, tenner patt-bru}. These had such interesting properties that we are motivated to look for boolean intervals in $\Star(n)$, as well. Note that this is a more general question than was studied for the Bruhat order, since we allow arbitrary intervals to be boolean, not just principal order ideals.

\begin{definition}
An interval $[\sigma,\pi]$ is \emph{boolean} if it is isomorphic to a boolean algebra.
\end{definition}

Due to Lemma~\ref{lem:single sliced cycle}, a study of boolean intervals in $\Star(n)$ amounts to understanding when the Catalan number $C_n$ is equal to $2^{n-1}$, and this happens only when $n \in \{1,2\}$. This and Corollary~\ref{cor:intervals are products of NC} characterize boolean intervals.

\begin{corollary}\label{cor:boolean intervals}
The interval $[\sigma,\pi]$ is boolean if and only if each cycle of $\sigma$ is sliced into at most two excerpts to form $\pi$.
\end{corollary}

This characterization allows us to enumerate several interesting things: 
\begin{itemize}
\item the number of boolean intervals with $\pi$ as maximal element is
$$\sum_{\substack{\text{involutions }\alpha \in \symm_m\\ i<j \text{ with } \alpha(i) = j}} \ell_i\ell_j;$$
\item the number of boolean intervals with $\pi$ as minimal element is
$$\prod_i\left(1 + \binom{\ell_i}{2}\right).$$
\end{itemize}
Thus, for example, the number of boolean intervals of the form $[\sigma,(1)(2)\cdots(n)]$ is the number of involutions in $\symm_n$, which is sequence A000085 of \cite{oeis}, and there are $1 + 3\cdot 1 = 4$ boolean intervals of the form $[\sigma,(123)(4)]$:
$$[(123)(4),(123)(4)] , \ 
[(1423),(123)(4)] , \
[(1243),(123)(4)] , \text{ and }
[(1234),(123)(4)].$$

\section{Future research}\label{sec:future}

Two main avenues for future research emerge from this work. The first it to establish how else star factorizations and noncrossing partitions interact, and to determine if that interaction can be leveraged in some way. Secondly, several of the results above exhibited a property of star factorizations that was independent of the pivot value. Thus we ask, how else can this pivot independence arise, and what does it mean in those settings?

\section{Acknowledgements}

I am grateful for the precious time and thoughtful suggestions given by an anonymous referee.

\end{document}